\documentclass[12pt]{article}
\usepackage[utf8]{inputenc}
\usepackage[a4paper,top=3cm,bottom=3cm,left=3cm,right=3cm]{geometry}
\usepackage{amsmath, amsthm, amsfonts}
\usepackage[all,cmtip]{xy}
\usepackage{natbib}
\newtheorem{teo}{\textbf{Theorem}}[section]
\newtheorem{dfn}[teo]{\textbf{Definition}}
\newtheorem{ejem}[teo]{\textbf{Example}}
\newtheorem{prop}[teo]{\textbf{Proposition}}

\newtheorem{lem}[teo]{\textbf{Lemma}}

\providecommand{\keywords}[1]
{
  \small	
  \textbf{\textit{Keywords---}} #1
}

\title{Reflections and Sheafifications in Algebraic and Topological Categories}
\author{Julio César Hernández Arzusa\\ Hernán Giraldo\\ Samir Rivero Castro}
\begin{document}
\maketitle
\begin{abstract}
    
In this article, we develop an explicit categorical realization of sheafification based on colimits, products, and subobjects, emphasizing its behavior in algebraic and topological-algebraic settings. We prove that if $\mathcal{C}$ is a reflective subcategory of a category $\mathcal{A}$, then the presheaf category $\mathbf{Psh}(X,\mathcal{C})$ is reflective in $\mathbf{Psh}(X,\mathcal{A})$. We further investigate the interaction between reflections and sheafification, obtaining natural conditions under which these  constructions are naturally isomorphic.

\end{abstract}


\thispagestyle{empty}

\section{Introduction}

Sheaf theory constitutes one of the central bridges between local and global phenomena in modern mathematics. Since its origins in algebraic topology and algebraic geometry, sheaves have become fundamental tools in homological algebra, category theory, logic, differential geometry, topological algebra, and mathematical physics. Their importance stems from the fact that they provide a categorical framework for encoding and reconstructing global structures from compatible local data.\\

From a categorical viewpoint, the passage from presheaves to sheaves is governed by the sheafification functor, which realizes the category of sheaves as a reflective subcategory of the category of presheaves. Classical treatments of sheafification can be found in \cite{godement1973topologie, tennison1975sheaf, maclane2012sheaves}, 
 while in \cite{kashiwara2006} was developed highly general existence results in exact and abelian categorical settings. However, many of these constructions remain abstract and are often difficult to adapt to topological-algebraic categories.\\

 In recent years, categorical sheaf theory has experienced significant development in several directions. Bunk studied homotopy sheaves on generalized spaces and analyzed how sheaf-theoretic structures behave under free cocompletions and Kan extensions, highlighting the deep interaction between presheaf categories and homotopical methods (see \cite{bunk2023homotopy}). Meanwhile, Barnes and Sugrue developed equivariant sheafification techniques for profinite groups, obtaining explicit constructions of equivariant presheaves and sheaves in topological-algebraic contexts (see \cite{barnes2022equivariant}).\\

 At the same time, recent advances have emphasized the role of reflective and coreflective subcategories in modern category theory. Structural characterizations of reflective subcategories in additive and pretriangulated settings were recently obtained in (see \cite{cortes2023reflective}), revealing new connections between reflection theory, localization, and categorical approximation methods. Related developments in higher and derived category theory include the study of reflective DG-categories and approximable triangulated categories (see \cite{goodbody2026approximable}).\\

 Another important recent direction concerns generalized sheaf representations and monoidal sheaf categories. Heunen and Barbosa proved representation theorems for monoidal categories via sheaves of local monoidal structures, extending classical Stone-type dualities and sheaf representations into a broad categorical setting (see \cite{barbosa2023sheaf}).\\

 Sheaves come naturally when you want to ``glue local information together to get global information", and this can be stated in a fairly general categorical context (see \cite{maclane2012sheaves}). The sheafification of a sheaf appears very generally in \cite{kashiwara2006}, the novel aspect of this work is provide a realization of the sheafification in terms of colimits, products, and subobjects.This approach does not make it possible to construct it in topological and topological-algebraic categories such as topological spaces, topological vector spaces, topological groups, and topological semigroups. We also provide an interaction between sheafifications  and reflections to obtain more reflective subcategories of a category presheaves.  Finding reflections is important because it allows us to approximate more complicated objects with those with desired properties in a universal and optimal way. In \cite{tkachenko2014axioms} several examples of reflective subcategories of the semitopological groups category appear. Whereas in \cite{castillo2021dual}, we can find examples of reflective subcategories of the category of topological groups.\\

Motivated by these advances, the purpose of this article is to study the interaction between sheafification and reflective subcategories in categories of presheaves with algebraic and topological-algebraic values. Our approach provides explicit realizations of sheafification using products, colimits, and subobjects, allowing the construction to be extended beyond classical algebraic categories into settings involving topological groups, compact Hausdorff spaces, compact semigroups, and related structures.

The paper is organized as follows.\\

In Section \ref{2}, we give the preliminaries concepts and in Section \ref{3},  we study reflective structures in categories of presheaves. More precisely, given a reflective subcategory $\mathcal{C}$ of a category $\mathcal{A}$, we prove that the presheaf category
$\mathbf{Psh}(X,\mathcal{C})$
is itself a reflective subcategory of
$\mathbf{Psh}(X,\mathcal{A})$.
We also describe explicitly the reflection functor induced at the presheaf level and analyze its universal properties.\\

Section \ref{1932} is devoted to the study of fibers of presheaves with values in algebraic and topological-algebraic categories. Using forgetful functors, we prove that the algebraic and topological structures of fibers can be recovered from their underlying set-theoretic constructions. In particular, we establish reconstruction results for presheaves with values in the categories  $\mathcal{GRP}$, $\mathcal{TOP}$, $\mathcal{TOPG}$,
and related algebraic categories. These results provide the technical foundation for the explicit construction of sheafifications developed later in the paper.\\

In Section \ref{5}, we construct sheafifications explicitly using products, colimits, and subobjects. We prove that this construction yields the sheafification functor in several algebraic categories, including groups, semigroups, modules, ordered structures, and related universal algebras. We then study the interaction between reflections and sheafification, obtaining natural conditions under which  $(r_\mathcal{C}F)^{+}$  and  $r_\mathcal{C}(F^{+})$
are naturally isomorphic. The final part of Section \ref{5} is devoted to compact topological and topological-algebraic categories. In particular, we analyze sheafifications in categories of compact Hausdorff spaces, compact topological semigroups, compact topological groups, and torsion-free compact groups. These examples illustrate how the categorical techniques developed throughout the article extend naturally to topological-algebraic settings where standard algebraic methods are no longer sufficient.

\section{Preliminaries}\label{2}

If $\mathcal{C}$ is a category, then $Ob(\mathcal{C})$ is the class of objects in $\mathcal{C}$. If $A,B\in Ob(\mathcal{C})$, then $Hom_{\mathcal{C}}(A,B)$ is the set of all morphisms from $A$ to $B$.
Let us denote by $\mathcal{SET}$, $\mathcal{GRP}$, $\mathcal{TOP}$ and $\mathcal{TOPG}$ the category of sets, groups, topological spaces, and topological groups, respectively.  If $(X, \tau)$ is a topological space, we denote by $\mathfrak{U}(X)$ the category whose objects are the open subsets of \( X \), and whose morphisms are the inclusion maps between open sets. That is, if \( U \) and \( V \) are open subsets of \( X \), there exists a morphism from \( U \) to \( V \) if and only if \( U \subseteq V \); in this case, the morphism is the inclusion of \( U \) into \( V \) ($\mathfrak{U}(X)$ is the category $\mathbf{PO}(\tau, \subseteq)$, see \cite[Example 2, page 3]{agore2023first}).\\
Let $X$ be a topological space and $\mathcal{C}$ a category, a presheaf of $X$ with values in $\mathcal{C}$ is a contravariant functor $F\colon \mathfrak{U}(X) \longrightarrow \mathcal{C}$. We denote by $\mathbf{Psh}(X, \mathcal{C})$ the category of presheaves of $X$ with values in $\mathcal{C}.$ If $U$ and $V$ are open sets in $X$, $U\subseteq V$ and $F$ is an object in $\mathbf{Psh}(X, \mathcal{C})$, there exists a unique $\mathcal{C}$-morphism $F_{U}^{V}\colon F(V)\longrightarrow F(U)$. Moreover, if $U,V,W$ are open sets in $X$ such that $U\subseteq V\subseteq W$, then $F_{U}^{V}F_{V}^{W}=F_{U}^{W}$.  Given $x\in X$, let us denote by $\mathcal{V}_x$, the set of open neighborhoods of $x$. If $\mathcal{C}$ has small colimits, let $F_x=Col_{V\in \mathcal{V}_x}F(V)$ (\( F_x \), which is called the fiber of \( F \) at \( x \), see \cite{tennison1975sheaf}), where $Col$  is the colimit. Then we have  the initial cocone  $\{\rho_{x, \mathcal{C}}^{F,V}\colon V\longrightarrow F_x\}_{V\in \mathcal{V}_x}$. Therefore, if $U, V\in \mathcal{V}_x$ and $U\subseteq V$, then the following diagram 

\begin{equation} \label{1931}\xymatrix{F(U)\ar[rr]^{\rho_{x,\mathcal{C}}^{U,F}} &  &  F_x\\
&   &\\
F(V)\ar[uurr]_{\rho_{x,\mathcal{C}}^{V,F}}\ar[uu]^{F_U^V} & }\end{equation} commutes.\\   We say $\mathcal{C}$ is a concrete category if the objects of $\mathcal{C}$ are sets endowed with some structures and the $\mathcal{C}$-morphism are functions that respect these structures.\\Let $\mathcal{C}$ be a concrete category. We say that $F\in Psh(X, \mathcal{C})$ is a sheaf of $X$ on  a concrete category $\mathcal{C}$ if it holds the following axioms for any open set \( U \subseteq X \) and any open cover \( \{ U_i \}_{i \in I} \) of \( U   \):

\begin{itemize}
    \item[1.] If $s_1, s_2\in F(U)$ and $F_{U_i}^U(s_1)=F_{U_i}^{U}(s_2)$ for all $i \in I$, then $s_1=s_2$.
    \item[2.] Suppose that we are given $(s_i)_{i\in I}\in \prod_{i\in I}F(U_i)$ such that $F_{U_i\cap U_j}^{U_i}(s_i)=F_{U_i\cap U_j}^{U_j}(s_j)$ for each $i,j\in I$. Then there exists $s\in F(U)$ such that  $F_{U_i}^{U}(s)=s_i$ for each $i \in I$.
\end{itemize}
We denote by $\mathbf{Sh}(X, \mathcal{C})$  the subcategory of $\mathbf{Psh}(X,\mathcal{C})$, whose objects are the sheaves.\\\\
Let $\mathcal{C}$ be a full subcategory of a category $\mathcal{A}$ and $A\in Ob(\mathcal{A})$.  A      $\mathcal{C}$-reflection of $A$ is $\mathcal{A}-$morphism $r_\mathcal{C}^A\colon A\longrightarrow r_\mathcal{C}A$, where $r_\mathcal{C}A\in Ob (\mathcal{C})$, such that for every $\mathcal{A}-$morphism $f\colon A\longrightarrow B$ with $B\in Ob (\mathcal{C})$, there exists a unique $\mathcal{C}-$morphism $r_\mathcal{C}(f)\colon r_\mathcal{C}A\longrightarrow B$ such that the following diagram 
$$\xymatrix{   A\ar[rr]^{r_\mathcal{C}^{A}} \ar[dd]^{f} &   &    r_\mathcal{C}A\ar[ddll]^{r_{\mathcal{C}}(f)}\\
&   &  \\
B}  $$ commutes.  It is easy to see that the $\mathcal{C}$-reflection is unique up to isomorphism. We say that $\mathcal{C}$ is a reflective subcategory of $\mathcal{A}$ if each $A\in Ob(\mathcal{A})$ has a $\mathcal{C}$-reflection. If $\mathcal{C}$ is a reflective subcategory of $\mathcal{A}$, then  give a $\mathcal{A}$-morphism, $f\colon A \longrightarrow B$, there exists a unique $\mathcal{C}-$morphism $r_\mathcal{C}(f)\colon r_\mathcal{C}A\longrightarrow r_\mathcal{C}B$
such that the following diagram 

$$\xymatrix{  A \ar[rr]^{f}\ar[dd]^{r_\mathcal{C}A}&    &   B\ar[dd]^{r_\mathcal{C}^{B}}\\
&    &\\
r_\mathcal{C}A\ar[rr]^{r_\mathcal{C}(f)}&     &r_\mathcal{C}B
}   $$ commutes. So $r\colon \mathcal{A}\longrightarrow \mathcal{C}$, defined by  $r_\mathcal{C}(f\colon A\longrightarrow B)=r_\mathcal{C}(f) \colon r_\mathcal{C}A\longrightarrow r_\mathcal{C}B,$ for each $\mathcal{A}$-morphism $f\colon A\longrightarrow B$, defines a  functor from $\mathcal{A}$ to $\mathcal{C}$, called the functor induced by reflection.\\\\

In \cite{kashiwara2006} it is proven that $\mathbf{Sh}(X, \mathcal{A})$ is a reflective subcategory of $\mathbf{Psh}(X, \mathcal{A})$ if $\mathcal{A}$ satisfies:
\begin{itemize}
    \item A admits small projective and small inductive limits.
    \item  Small filtrant inductive limits are exact.
    \item  A satisfies the IPC-property (see \cite[Defintion 3.1.10]{kashiwara2006}).
\end{itemize}


\section{Presheaves in reflective categories}\label{3}
Let $\mathcal{C}$ be a reflective subcategory of a category $\mathcal{A}$, $X$ a topological space and $F\in \mathbf{Psh}(X,\mathcal{A})$. Let $U$ and $V$ be open sets in $X$, with $U\subseteq V$. Then there exists a unique $\mathcal{C}$-morphism $r_\mathcal{C}(F_U^V)\colon r_\mathcal{C}F(U) \longrightarrow r_\mathcal{C}F(V)$ such that the following diagram $$\xymatrix{F(V)\ar[rr]^{F_U^{V}} \ar[dd]^{r_\mathcal{C}^{F(V)}} &     & F(U)\ar[dd]^{r_\mathcal{C}^{F(U)}} \\ &  &\\
r_\mathcal{C}F(V)\ar[rr]^{r_\mathcal{C}(F_U^V)}&    & r_\mathcal{C}F(U)\\
}    $$ commutes. Therefore, $r_\mathcal{C}F$ is a presheave in  $\mathbf{Psh}(X,  \mathcal{C})$. If we define $\theta_U^{F,\mathcal{C}}\colon F(U)\longrightarrow r_\mathcal{C}F(U)$, by $\theta_U^{F,\mathcal{C}}=r_\mathcal{C}^{F(U)}$, then we have that $\theta^{F,\mathcal{C}}\colon F\longrightarrow  r_\mathcal{C}F$ defines a natural transformation. Let $\theta\colon F\longrightarrow G$  be a natural transformation, where $G\in \mathbf{Psh}(X, \mathcal{C})$. For each $U$ open in $U$ there is a unique $\mathcal{C}$-morphism ${C}(\theta)_U\colon r_\mathcal{C}F(U)\longrightarrow r_\mathcal{C}G(U)$, such that  the following diagram 

$$\xymatrix{  F(U)\ar[rr]^{\theta_U}\ar[dd]^{r_\mathcal{C}^{F(U)}} &    &    G(U)\ar[dd]^{r_{\mathcal{C}}^{G(U)}}\\
&     &\\
r_\mathcal{C}F(U)\ar[rr]^{\mathcal{C}(\theta)_U} &    &r_\mathcal{C}G(U)}    $$ 
commutes. This proves that $\{\mathcal{C}(\theta)_U\colon r_\mathcal{C}F(U)\longrightarrow r_\mathcal{C}G(U): U \mbox{ is open in $X$}\}$ is a natural transformation and $\mathcal{C}(\theta)\circ \theta^{F,\mathcal{C}}=\theta.$  This can be summarized in the following theorem.

\begin{teo} \label{253}Let $\mathcal{C}$  be a reflective subcategory of $\mathcal{A}$.
    Let $F\in \mathbf{Psh}(X, \mathcal{A})$. Then $r_\mathcal{C}F\in \mathbf{Psh}(X,\mathcal{C})$, $\theta^{F,\mathcal{C}}\colon F\longrightarrow r_\mathcal{C}F$ is a natural transformation, such that for every natural transformation $\theta \colon F\longrightarrow  G$, where $G\in \mathbf{Psh}(X,\mathcal{C}) $, there exists a unique natural transformation $\mathcal{C}(\theta)\colon r_\mathcal{C}F\longrightarrow G$ such that $\mathcal{C}(\theta)\circ \theta^{F,\mathcal{C}}=\theta.$ Therefore, $\mathbf{Psh}(X, \mathcal{C})$ is a reflective subcategory of $\mathbf{Psh}(X, \mathcal{A})$.
\end{teo}

\section{Presheaves with values in $\mathcal{SET}$, $\mathcal{TOP}$, $\mathcal{GRP}$ and $\mathcal{TOPG}$} \label{1932}

In this section, we study the behavior of presheaves when we apply forgetful functors on $\mathcal{TOP}$, $\mathcal{GRP}$ or $\mathcal{TOPG}$.  Let  us note by  $U_1$ (forgets the topology structure) and $U_2$ (forgets the group structure) the forgetful functors   from $\mathcal{TOPG}$  to  $\mathcal{GRP}$ and  from $\mathcal{GRP}$ to $\mathcal{SET}$, respectively. 

The following proposition follows  from \cite[Propositions 4.2 and 3.1]{tennison1975sheaf}.

\begin{prop}\label{121024}
     Let $x\in X$. If $F\in \mathbf{Psh}(X, \mathcal{SET})$, then \begin{itemize}
\item[a)]  If $t\in F_x$, then there exists $U\in \mathcal{V}_x$ and $a\in F(U)$ such that $\rho_{x,\mathcal{SET}}^{U,F}(a)=t$.
\item[b)]  If $U,V\in \mathcal{V}_x$, $a\in F(U)$ and $b\in F(V)$. Then $\rho_{x,\mathcal{SET}}^{U,F}(a)=\rho_{x,\mathcal{SET}}^{V,F}(b)$ if and only if there exists an open set $W$ in $X$, with $W\subseteq U\cap V$ such that $F_W^{U}(a)=F_W^{V}(b)$.
\item[c)] If $F$ is a sheaf and $s_1, s_2\in F(U)$, where $U$ is open in $X$, then  $s_1=s_2$ if and only if $\rho_{x, \mathcal{SET}}^{U,F}(s_1)=\rho_{x, \mathcal{SET}}^{U,F}(s_2)$ for each $x\in U$.
    \end{itemize}
\end{prop}



    Let $F$ be a presheave with values in  $\mathcal{TOPG}$, we can apply forgetful functors appropriately, so that it results in a presheaf $F_1$ in $\mathcal{SET}$. We will prove that it is possible to recover the structures, both algebraic and topological on $(F_1)_x$, such that it is topologically  isomorphic to $F_x$. A similar result is provided in \cite{godement1973topologie}, but only for groups.  In \cite{hernandez2020reflections}  an example is given of how algebraic semitopological structures are recovered after applying reflections on the  topological structure.

\begin{teo}\label{263}
    Let $F\in \mathbf{Psh}(X, \mathcal{GRP})$. For each $x\in X$, there exists a binary operation $\ast$ in $(U_2F)_x$  such that $((U_2F)_x, \ast)$ is a group isomorphic to $F_x$.
\end{teo}

\begin{proof}
 First, we define $\ast$. Given $x_1,y\in U_2(F_x)$, by Proposition \ref{121024}, we can find $U,V\in \mathcal{V}_x$ such that $(a,b)\in U_2(F(U))\times U_2(F(V))$, $\rho_{x,\mathcal{SET}}^{U,U_2F}(a)=x_1$ and $\rho_{x,\mathcal{SET}}^{V, U_2F}(b)=y$. Let us define  
\begin{equation}
    x_1\ast y :=\rho_{x, \mathcal{SET}}^{U\cap V,U_2F}(F_{U\cap V}^{U}(a)F_{U\cap V}^{V}(b)).
\end{equation}

\textit{\textbf{Claim 1.}}  The operation $\ast$ is well defined.\\ Indeed,  suppose that $a=b\in (U_2F)_x$, by Proposition \ref{121024} there exist $U,V\in \mathcal{V}_x$ and $s_a$ y $s_b$ in $U_2F(U)$ and $U_2F(V)$, respectively, such that $a=\rho_{x,\mathcal{SET}}^{U, U_2F}(s_a)$ and $b=\rho_{x,\mathcal{SET}}^{V, F
(\mathcal{SET})}(s_b)$. Then, we can find an open set in $X$, $W$ such that $F_W^U(s_a)=F_W^{V}(s_b)$. Let $T\in \mathcal{V}_x$ and $s_c\in F(T)$, such that $c=\rho_{x, \mathcal{SET}}^{T,U_2F}(s_c)$. By definition, $a\ast c=\rho^{U\cap T, U_2F}_{x,\mathcal{SET}}(F_{U\cap T}^U(s_a)F_{U\cap T}^T(s_c))$ and $b\ast c=\rho^{V\cap T\cap U, U_2F}_{x, \mathcal{SET}}(F_{V\cap T}^V(s_B)F_{V\cap T}^T(s_c))$. To see that $a\ast c=b\ast c$, it is enough to prove that $\rho^{U\cap T, U_2F}_{x, \mathcal{SET}}(F_{U\cap T}^U(s_a)F_{U\cap T}^T(s_c))=\rho^{V\cap T\cap U,U_2F}_{x,\mathcal{SET}}(F_{V\cap T}^V(s_B)F_{V\cap T}^T(s_c))$. Note that $W\cap T\subseteq (U\cap T)\cap (V\cap T)$ and 

\begin{equation*}
    \begin{split}
        (U_2F)_{W\cap T}^{U\cap T}((U_2F)_{U\cap T}^U(s_a)(U_2F)_{U\cap T}^T(s_c))&=F_{W\cap T}^{U\cap T}(F_{U\cap T}^U(s_a)F_{U\cap T}^T(s_c))\\
        &=F_{W\cap T}^{U\cap T}(F_{U\cap T}^U(s_a))F_{W\cap T}^{U\cap T}(F_{U\cap T}^T(s_c))\\
        &=F_{W\cap T}^{U}(s_a)F_{W\cap T}^{T}(s_c)\\
        &=(U_2F)_{W\cap T}^{U}(s_a)(U_2F)_{W\cap T}^{T}(s_c).
    \end{split}
\end{equation*}

 Similarly,  $(U_2F)_{W\cap T}^{V\cap T}((F( \mathcal{SET}))_{V\cap T}^V(s_B)(U_2F)_{V\cap T}^T(s_c))=(U_2F)_{W\cap T}^{V}(s_a)(U_2F)_{W\cap T}^{T}(s_c)$. But $$(U_2F)_{W\cap T}^{U}(s_a)=(U_2F)_{W\cap T}^{W}((U_2F)_W^{U}(s_a))=(U_2F)_{W\cap T}^{W}((U_2F)_W^{V}(s_b))=(U_2F)_{W\cap T}^{V}(s_b).$$

 \begin{equation*}
    \begin{split}
    (U_2F)_{W\cap T}^{U\cap T}((U_2F)_{U\cap T}^U(s_a)(U_2F)_{U\cap T}^T(s_c))&= F_{W\cap T}^{U\cap T}(F_{U\cap T}^U(s_a)F_{U\cap T}^T(s_c))\\
    &=F_{W\cap T}^{U\cap T}(F_{U\cap T}^U(s_a))F_{W\cap T}^{U\cap T}(F_{U\cap T}^T(s_c)))\\
    &=F_{W\cap T}^{U}(s_a)F_{W\cap T}^{T}(s_c)\\
    &=(U_2F)_{W\cap T}^{U}(s_a)(U_2F)_{W\cap T}^{T}(s_c)\\
    &=(U_2F)_{W\cap T}^{V}(s_b)(U_2F)_{W\cap T}^{T}(s_c)\\
    &=F_{W\cap T}^{V}(s_b)F_{W\cap T}^{T}(s_c)\\
    &=F_{W\cap T}^{V\cap T}(F_{V\cap T}^V(s_b)) F_{V\cap T}^{T}(s_c))\\
    &=(U_2F)_{W\cap T}^{V\cap T}(F_{V\cap T}^V(s_b)) F_{V\cap T}^{T}(s_c)).
\end{split} 
\end{equation*}

Thus, by Proposition \ref{121024}, we have that $$\rho^{U\cap T, U_2F}_{x, \mathcal{SET}}(F_{U\cap T}^U(s_a)F_{U\cap T}^T(s_c))=\rho^{V\cap T,U_2F}_{x,\mathcal{SET}}(F_{V\cap T}^V(s_B)F_{V\cap T}^T(s_c)).$$ Analogously we can prove that $c\ast a=c\ast b$, what it proves that $\ast$ is well defined.\\
\textbf{\textit{Claim 2.}}  $\ast$ is associative. It is not hard to prove that if $a=\rho_{x,\mathcal{SET}}^{U,U_2F}(s_a)$, $b=\rho_{x,\mathcal{SET}}^{V,U_2F}(s_b)$ and $c=\rho_{x,\mathcal{SET}}^{W,U_2F}(s_c)$, with $U,V,W\in \mathcal{V}_x$, then $$a \ast (b \ast c)=\rho^{U\cap V\cap W, U_2F}_{x,\mathcal{SET}}(F_{U\cap V\cap W}^{U}(s_a)F_{U\cap V\cap W}^{V}(s_b)F_{U\cap V\cap W}^{W}(s_c))=(a \ast b) \ast c.$$\\
Let us denote by $e_V$ the neutral element of $F(V)$.\\

\textbf{\textit{Claim 3.}} If $V,U\in \mathcal{V}_x$, $\rho_{x,\mathcal{SET}}^{V,U_2F}(e_V)=\rho_{x,\mathcal{SET}}^{U,U_2F}(e_U)$. It follows from the fact that $$F_{U\cap V}^U(e_U)=e_{U\cap V}=F_{U\cap V}^{V}(e_V).
$$\\
\textbf{\textit{Claim 4.}} If $a\in (U_2F)_x$ and $U\in \mathcal{V}_x$,  $a\ast \rho_{x,\mathcal{SET}}^{U,U_2F}(e_U)=\rho_{x,\mathcal{SET}}^{U,U_2F}(e_U)a=a$. Indeed, suppose that $a=\rho_{x,\mathcal{SET}}^{V,U_2F}(s_a)$, where $V\in \mathcal{V}_x$ and $s_a\in (U_2F)(V) $. Then $a\rho_{x,\mathcal{GRP}}^{U,F}(e_U)=\rho_{x,\mathcal{SET}}^{V\cap U,U_2F}(F_{U\cap V}^V(s_a)F_{U\cap V}^U(e_U))=\rho_{x,\mathcal{SET}}^{V\cap U,U_2F}(F_{U\cap V}^V(s_a)e_{U\cap V})=\rho_{x,\mathcal{SET}}^{V\cap U,U_2F}(F_{U\cap V}^V(s_a))$. Note that $F_{U\cap V}^{U\cap V}(F_{U\cap V}$. Now, $F_{U\cap V}^{U\cap V}(F_{U\cap V}^{V}(s_a))=F_{U\cap V}^{V}(s_a)$, this proves that $$a\rho_{,\mathcal{SET}}^{U,U_2F}(e_U)=\rho_{U\cap V}^{x}(F_{U\cap V}^V(s_a))=\rho_{x,\mathcal{SET}}^{V,U_2F}=a.$$
Similarly, we have that $\rho_{x,\mathcal{SET}}^{U,U_2F}(e_U) \ast a=a$. \\

\textbf{\textit{Claim 5.}} If $a=\rho_{,\mathcal{SET}}^{V,U_2F}(s_a)$, where $V\in \mathcal{V}_x$ and $s_a\in F(V) $, then $a\rho_{x,\mathcal{SET}}^{V,U_2F}(s_a^{-1})=\rho_{x, \mathcal{SET}}^{V, U_2F}(s_a^{-1})a=\rho_{x,\mathcal{SET}}^{V,U_2F}(e_V)$. Indeed, $$a\rho_{x, \mathcal{SET}}^{V, U_2F}(s_a^{-1})=\rho_{x, \mathcal{SET}}^{V, U_2F}(F_V^{V}(s_a)F_V^{V}(s_a^{-1})=\rho_{x,\mathcal{SET}}^{V,U_2F}(s_as_a^{-1})=\rho_{x,\mathcal{SET}}^{V,U_2F}(e_V).$$ This proves that $a^{-1}=\rho_{x,\mathcal{SET}}^{V,U_2F}(s_a^{-1})$.\\
 By claims 1, 2, 3, 4 and 5, we have that  $((U_2F)_x, \ast)$ is a group. \\Now, let   $s_a,s_b\in F(U)$, put $a=\rho_{x,\mathcal{SET}}^{U,U_2F}(s_a)$ and $b=\rho_{x,\mathcal{SET}}^{U,U_2F}(s_b)$. Then $\rho_{x,\mathcal{SET}}^{U,U_2F}(s_a )\ast \rho_{x,\mathcal{SET}}^{U,U_2F}(s_b)=ab=\rho_U(F_U^{U}(s_a)F_U^{U}(s_b))=\rho_U^{x}(s_as_b)$. It means $=\rho_{x,\mathcal{SET}}^{U,U_2F}$ is a homomorphism for each $U\in \mathcal{V}_x$. The universal property implies that there exists a unique $\mathcal{GRP}$-morphism $f\colon F_x\longrightarrow ((U_2F)_x, \ast) $  such that for each $U\in \mathcal{V}_x$, the following diagram 
 $$\xymatrix{ F(U) \ar[rr]^{\rho_{x,\mathcal{SET}}^{U,U_2F}} \ar[dd]_{\rho_{x, \mathcal{GRP}}^{U,F}} &    &   ((U_2F)_x,\ast) \\
 &    & \\
 F_x\ar[uurr]_f &  
 }
$$ commutes. Similarly, we can find a unique $\mathcal{SET}$-morphism $g\colon (U_2F)_{x}\longrightarrow U_2(F_x$) such that the following diagram $$\xymatrix{  F(U) \ar[rr]^{\rho_{x,\mathcal{SET}}^{U,U_2F}} \ar[dd]_{\rho_{x, \mathcal{GRP}}^{U,F}} &    &   (U_2F)_x\ar[ddll]_g \\
 &    & \\
U_2( F_x)
}$$ commutes. It is not hard to prove that $g\circ U_2(f)$ is the identity morphism in $Hom_\mathcal{SET}(U_2(F_x), U_2(F_x))$ and $U_2(f)\circ g$ is the identity morphism in $Hom_{\mathcal{SET}}((U_2F)_x,(U_2F)_x)$. It means $f$ is a biyective homomorphism of groups, what it  proves that $((U_2F)_x, \ast)$ is isomorphic to $F_x$.
\end{proof}

The last statement in the proof of Theorem \ref{263} tells us that the  forgetful functor $U_2$ preserves the colimit $F_x^{\mathcal{GRP}}$, given that $(U_2F)_x^{\mathfrak{S}}$ is isomorphic to $U_2(F_x^\mathcal{GRP})$, but in general this is no longer true, it is because $U_2$ does not preserve coproducts (see \cite[Proposition 2.5.3]{agore2023first}).\\
Although the theorem above was stated only for groups, it can be extended to other universal algebras such as rings, semigroups, modules, etc.

\begin{teo}
If  $F\in Psh(X, \mathcal{TOPG})$ and $x\in X$, then  $(U_2U_1F)_x$ admits a topology $\tau$ making  it into a topological group, such that each $U\in \mathcal{V}_x$, $\rho_{x, \mathcal{GRP}}^{U, U_1F} \colon F(U)\longrightarrow ((U_2U_1F)_x,\tau)$ is a continuous homomorphism. Moreover $((U_2U_1F)_x,\tau)$ is 
If  $F\in Psh(X, \mathcal{TOPG})$, then  $(U_2U_1F)_x$ admits a topology $\tau$ making  it into a topological group, such that each $U\in \mathcal{V}_x$, $\rho_{x, \mathcal{GRP}}^{U, U_1F} \colon F(U)\longrightarrow ((U_2U_1F)_x,\tau)$ is a continuous homomorphism. Moreover, $((U_2U_1F)_x,\tau)$ is topologically isomorphic to $F_x$.
\end{teo}

\begin{proof}
    There exists a unique homomorphism $f\colon (U_1F)_x\longrightarrow U_1(F_x)$ such that, for each $U\in \mathcal{V}_x$, the following diagram 
    $$ \xymatrix{(U_1F(U)\ar[rr]^{\rho_{x,\mathcal{GRP}^{U, U_1F}}}  \ar[dd]_{U_1(\rho_{x,\mathcal{TOPG}}^{U, F})}&   & (U_2U_1F)_x\ar[ddll]_f\\ &  &\\
   U_1(F_x) &   & \\
      }$$ commutes. Let $\tau$ be the initial topology induced by $f\colon (U_2U_1F)_x \longrightarrow F_x$, it is known that $((U_2U_1F)_x,\tau)$ is a topological group, also since $f\circ \rho_{x,\mathcal{GRP}^{U, U_1F}}=\rho_{x,\mathcal{TOPG}}^{U, F})$ and $\rho_{x,\mathcal{TOPG}}^{U, F})$ is continuous, we have that $\rho_{x,\mathcal{GRP}^{U, U_1F}}$ is continuous. From the universal property of  colimits, there exists a unique continuous homomorphism $g\colon F_x\longrightarrow ((U_2U_1F)_x,\tau)$ such that, for each $U\in \mathcal{V}_x$, following diagram 
      $$\xymatrix{F(U)\ar[rr]^{\rho_{x, \mathcal{TOPG}^{U,F}}} \ar[dd]_{\rho_{x, \mathcal{GRP}}^{U, U_1F}}&   &  F_x\ar[ddll]_g\\
      &   & \\
      ((U_2U_1F)_x,\tau) 
      }$$  commutes. \\It is not hard to prove that $g\circ f$ is the identity morphism in $Hom_\mathcal{TOPG}( ((U_2U_1F)_x,\tau), ((U_2U_1F)_x,\tau))$ and $g\circ f$ is the identity morphism in $Hom_{\mathcal{TOPG}}(F_x^{\mathcal{GRP}}, F_x^{\mathcal{GRP}})$, what it proves that $f$ is a topological  isomorphism. 
\end{proof}
  The proposition \ref{121024} is valid for $\mathbf{Psh}(X, \mathcal{SET})$, however, the results of this section show that this can also be applied to any full subcategory of full subcategory of  $\mathcal{TOP}$, $\mathcal{GRP}$ or $\mathcal{TOPG}$. with small products and colimits.\\

\section{sheafification of presheaves with values in $\mathcal{SET}$ and some algebraic categories}\label{5}

  The existence of sheafifications  is studied in \cite{kashiwara2006}  for more general categories than those given in this section. We use products and sub-objects for this purpose, as this allows us to see when this process results in a sheaf  with values within the category in question. \\ 
 In this section $\mathcal{A}$ is a concrete category that satisfies the following properties: 
 \begin{itemize}
     \item Every monomorphism in $\mathcal{A}$ that is an epimorphism  is  an isomorphism.
     \item  It is closed under taking small products and small colimits.
 \end{itemize}

Let $F\in \mathbf{Psh}(X, \mathcal{A})$. Given $U$ open in $X$, define $F^{+}(U)$ to be the set of  points $(s_x)_{x\in U}\in \prod_{x\in U}F_x$ such that for every $x\in U$ there exist $ V \in \mathcal{V}_x$, with $V\subseteq U$ and $t\in F(V)$, satisfying $s_y=\rho_{y, \mathcal{C}}^{V,F}(t)$ for any $y\in V$. It is not difficult to prove that if $V\subseteq U$ and $(s_x)_{x\in U}\in F^{+}(U)$, then $(s_x)_{x\in V}\in F^{+}(V)$. Therefore, we can define $(F^{+})_V^{U}\colon F^{+}(U)\longrightarrow  F^{+}(V)$, by $(F^{+})_V^{U}((s_x)_{x\in U})=(s_x)_{x\in V}$, for every $(s_x)_{x\in U}\in F^{+}(U)$. \\ Given $U$ open in $X$ and $F\in \mathbf{Psh}(X, \mathcal{C})$.  There exists a $\mathcal{C}$-morphism   $p_U^{F}\colon F(U)\longrightarrow \prod_{x\in U}F_x$ (or $p_U$ when there is no danger of confusion) such that $\pi_x^U\circ p_U=\rho_{x, \mathcal{C}}^{U,F}$, where$\pi_x^U$ the $x$-th denotes the projection from $\prod_{x\in U} F_x$ to $F_x$. Clearly, $p_U(s)\in F^{+}(U)$ for each  $s\in F(U)$.\\

\begin{lem}\label{193}
    Let  $\mathcal{A}$ be  a subcategory of $\mathcal{SET}$  or  $\mathcal{GRP}$ closed under taking subobjects and  $F\in \mathbf{Psh}(X, \mathcal{A})$. Then $F^{+}(U)\in Ob(\mathcal{A})$ for each open set $U\subseteq X$.
\end{lem}

\begin{proof}
    If $\mathcal{A}$ is a subcategory of $\mathcal{SET}$ , $F^{+}(U)$ is a subset of $\prod_{x\in U}F_x$. Let  $\mathcal{A}$ be  a subcategory of $\mathcal{GRP}$, let us see that $F^{+}(U)$ is a subgroup of $\prod_{x\in U}F_x$. Indeed, let $(s_x)_{x\in U}, (r_x)_{x\in U}\in F^{+}(U)$. Given $x\in U$ there exist $V,W\in \mathcal{V}_x$, $V,W\subseteq U$ and $t\in F(V)$, $s\in F( W)$ such that $\rho_{y,\mathcal{A}}^{V,F}(t)=s_y$ for every $y\in V$ and $\rho_{y,\mathcal{A}}^{W,F}(s)=r_y$ for every $y\in W$. Let $R=W\cap V$.  Let us put  $m = F^{R}_{V}(t)  F^{R}_{w}(s).$ for every  $y \in R$,
\begin{align*}
\rho^{R, F}_{y,\mathcal{A}}(m)
&= \rho^{R, F}_{y,\mathcal{A}}(F^{V}_{R}(t)  F^{W}_{R}(s)) \\
&=\rho^{R, F}_{y,\mathcal{A}}(F^{V}_{R}(r))\rho^{R,F}_{y,\mathcal{A}}(F^{W}_{R}(t_2))\\
&=\rho^{V,F}_{y, \mathcal{A}}(t)\rho^{W,F}_{y,\mathcal{A}}(s)\\
&=s_y r_y=(sr)_y.
\end{align*}

Now, $t^{-1}\in F(V)$ and for every $y\in V$, $\rho_{y,\mathcal{A}}^{V,F}(t^{-1})=s_y^{-1}=(s^{-1})_y$. This proves $((s_x)_{x\in U})^{-1}\in F^{+}(U)$.

\end{proof}

From lemma above, we have that if $F\in \mathbf{Psh}(X, \mathcal{GRP})$, then  $F^{+}\in \mathbf{Sh}(X,\mathcal{GRP})$, but this result remains valid in some other   universal algebras such as rings, semigroups, vectorial spaces, modules, etc.

\begin{prop}
     Let $\mathcal{A}$  be  a subcategory of  $\mathcal{SET}$  or  $ \mathcal{GRP}$ closed under taking subobjects  and  $F\in \mathbf{Psh}(X, \mathcal{A})$. Then $F^{+}$  is a sheaf with values in   $\mathcal{A}.$ 
\end{prop}

\begin{proof}By Lemma \ref{193} it is enough to prove that axioms 1 and 2 hold.

\begin{itemize}
    \item[1.]  Let $U$ be open in $X$ and $U=\bigcup_{i\in I}U_i$ an open cover of $U$. Let us suppose that $$(F^{+})_{U_i}^{U}((s_x)_{x\in U})=(F^{+})_{U_i}^{U}((r_x)_{x\in U}),$$ where $(s_x)_{x\in U}, (r_x)_{x\in U}\in F^{+}(U)$. Let $p\in U$, then there exists $i\in I$ such that $p\in U_i$. By our assumption, $(F^{+})_{U_i}^{U}((s_x)_{x\in U})=(F^{+})_{U_i}^{U}((r_x)_{x\in U})$, that is $(s_x)_{x\in U_i}=(r_x)_{x\in U_i}$, therefore, $s_p=r_p$ for each $p\in U$, then $(s_x)_{x\in U}=(r_x)_{x\in U}$.

    \item[2.]  Suppose that $U$ is open in $X$ and $U=\bigcup_{i\in I}U_i$ is an open cover of $U$. Suppose that we are given $(s_i)_{i\in I}\in \prod_{i\in I}F^{+}(U_i)$ such that $(F^{+})_{U_i\cap U_j}^{U_i}(s_i)=(F^{+})_{U_i\cap U_j}^{U_j}(s_j)$, for each $i,j\in I$. If $x\in U$, there exists $i\in I$ such that $x\in U_{i}$. Let us put $s_x=(s_{i})_x$ (the $x$-th coordinate of $s_{i}$). $s_x$ is uniquely determined by $x$, since if $x\in U_i\cap U_j$, then by our assumption $((s_i)_x)_{x\in U_i\cap U_j}=((s_j)_x)_{x\in U_i\cap U_j}$, so that $(s_i)_x=(s_j)_x$.  Let $s=((s_{i})_x)_{x\in U}$, then $(F^{+})_{U_i}^{U}(s)=((s_i)_x)_{x\in U_i}=s_i$.
\end{itemize}  
\end{proof}

In \cite[Corollary 4]{maclane2012sheaves} it was proven that $\mathbf{Sh}(X, \mathcal{A})$ is a reflective subcategory  of $\mathbf{Psh}(X, \mathcal{A})$, when $\mathcal{A}=\mathcal{SET}$. In \cite{gray1965sheaves} a similar result was found, when $X$ is a discrete space and $\mathcal{A}$ is a category  with products. We will provide a realization of  sheafification of a presheaf using products and subobjects.

\begin{lem}\label{223} Let $\mathcal{A}$  be a subcategory of $\mathcal{SET}$ or  $\mathcal{GRP}$ closed under taking subobjects. Then $\{p_U\colon F(U)\longrightarrow F^{+}(U): U \mbox{is open in X}\}$ is a natural transformation for each $F\in \mathbf{Psf}(X, \mathcal{A})$. Moreover, if $F\in \mathbf{Sh}(X, \mathcal{A})$, then  $\{p_U\colon F(U)\longrightarrow F^{+}(U): U \mbox{is open in X}\}$ is a natural isomorphism. 
    
\end{lem}

\begin{proof} By equation \ref{1931},  it is easy to prove  that if $U$ and  $V$ are open in $X$, $V\subseteq U$, then
the following diagram

$$\xymatrix{ F(U)\ar[rr]^{p_U}\ar[dd]_{F_V^{U}} &   &  F^{+}(U)\ar[dd]^{(F^{+})_V^{U}}\\
&   &\\
F(V)\ar[rr]_{p_V} &   &  F^{+}(V)}$$ commutes. Let $F\in \mathbf{Sh}(X, \mathcal{C})$ and let us see that each $p_U$ is an isomorphism. The injectivity of $p_U$ follows from $c)$ Proposition \ref{121024}, although it is given for $\mathcal{SET}$, Section \ref{1932} tells us that it is also true for  $\mathcal{GRP}$. Let us see that $p_U$ is surjective. Indeed, let $(s_x)_{x\in U}\in F^{+}(U)$. For each $x\in U$ there exist $\mathcal{V}_x\ni V_x\subseteq U$ and $t_x\in F(V_x)$ such that $\rho_{y, \mathcal{A}}^{V_x,F}(t_x)=s_y$, for each $y\in V_x$. Now $U=\bigcup_{x\in U}V_x$ is an open cover of $U$, if $x,r \in U$, for every $y\in V_x \cap V_r$, we have that $\rho_{y, \mathcal{A}}^{V_x\cap V_r,F}(F_{V_x\cap V_r}^{V_x}(t_x))=\rho_{y, \mathcal{A}}^{V_x, F}(t_x)=s_y=\rho_{y, \mathcal{A}}^{V_r,F}(t_r)=\rho_{y, \mathcal{A}}^{V_r\cap V_x,F}(F_{V_x\cap V_r}^{V_r}(t_r))$, if we apply Proposition \ref{121024} $c$ again, we have that $F_{V_x\cap V_r}^{V_x}(s_x)=F_{V_x\cap V_r}^{V_r}(s_r)$, therefore, by axiom 2, there exists $s\in F(U)$ such that $F_{V_x}^{U}(s)=t_x$ for each $x\in U$. So that $p_U(s)=(\rho_{x, \mathcal{A}}^{U, F}(s))_{x\in U}=(\rho_{x\mathcal{A}}^{V_x,F}(F_{V_x}^{U}(s))_{x\in U}=(\rho_{x\mathcal{A}}^{V_x,F}(t_x))_{x\in U}=(s_x)_{x\in U}$.  Therefore, $p_U$ is bijective, which implies that $p_U$ is an isomorphism.

\end{proof}

\begin{prop}\label{2331}  Let $\mathcal{A}$ be a subcategory of  $\mathcal{SET}$ or $\mathcal{GRP}$ closed under taking subobjects, $F,G\in \mathbf{Psh}(X, \mathcal{A})$ and $\theta\colon F\longrightarrow G$ a natural transformation. Then there exists a unique natural transformation  $\theta^+\colon F^{+}\longrightarrow G^+$ such that  the following diagram 

\begin{equation}\label{233}\xymatrix{F(U)\ar[rr]^{\theta_U} \ar[dd]^{p_U^{F}} &   &   G(U)\ar[dd]^{p_U^{G}}\\
            &    &\\
            F^{+}(U)\ar[rr]^{\theta^{+}_U}  &    &   G^{+}(U)}  \end{equation}
commutes for each $U$ open in $X$.
\end{prop}

\begin{proof}
    If $U$ and $V$ are open in $X$ and $x\in V\subseteq U$, then the following diagram 

    $$ \xymatrix{  F(U)\ar[rr]^{\rho_{x,\mathcal{A}}^{U,G}\circ \theta_U} \ar[dd]^{F_V^{U}}&   &  G_x\\
    &    & \\
    F(V)\ar[uurr]_{\rho_{x,\mathcal{A}}^{V,G} \circ\theta_V } }$$
    commutes. Therefore, for each $x\in U$ there exists a unique $\mathcal{C}-$morphism $m_x\colon F_x\longrightarrow G_x$ such that the following diagram 

    $$\xymatrix{  F_x \ar[rr]^{m_x}&   &    G_x\\&  &\\
    F(U)\ar[uu]^{\rho_{x,\mathcal{A}}^{U,F}}\ar[uurr]_{\rho_{x,\mathcal{A}}^{U,G}\circ \theta_U}&   &}  $$
commutes for each $U\in \mathcal{V}_x$. So that, we can define $\theta^{+}_U\colon \prod_{x\in U}F_x\longrightarrow \prod_{x\in U}G_x$ such that $\theta^{+}_U((s_x)_{x\in U})=(m_x(s_x))_{x\in U}$. If $(s_x)_{x\in U}\in F^{+}(U)$ if $x\in U$, there exist $\mathcal{V}\ni V\subseteq U$ and $t\in F(V)$ such that $s_y=\rho_{y, \mathcal{A}}^{V,F}(t)$. $\theta_V(t)\in G(V)$ and $m_y(s_y)=m_y(\rho_{y,\mathcal{A}}^{V,F}(t_x))=\rho_{x,\mathcal{A}}^{V,G}(\theta_V(t^{+}))$. This shows that $\theta_U^{+}$ maps $F^{+}(U)$ in $G^{+}(U)$. It is straightforward to prove that the diagram \ref{233} commutes.

\end{proof}

The following theorem tells us that if $F\in \mathbf{Psh}(X, \mathcal{A})$, then $F^{+}$ is the reflection (the sheafification) of $F$ in $\mathbf{Sh}(X, \mathcal{A})$ 

\begin{teo}\label{2531}
    Let $F\in \mathbf{Psh}(X, \mathcal{A})$, where   $\mathcal{A}$ is a subcategory of $\mathcal{SET}$ or $\mathcal{GRP}$. closed under taking subobjects.  If $G\in \mathbf{Sh}(X, \mathcal{A})$ and $\theta\colon F\longrightarrow G$ is a natural transformation, then there exists a unique  natural transformation $\sigma_\theta\colon F^{+}\longrightarrow G$ such that $\sigma_\theta\circ p^{F}=\theta$.
    
\end{teo}

\begin{proof}
 By Proposition \ref{2331} there exists  a unique natural transformation 
 $\theta^{+}\colon F^{+} \longrightarrow G^+$ such that $p^{G}\circ \theta=\theta^{+}\circ p^{F}$. By Lemma \ref{223}, $p^{G}$ is a natural isomorphism, therefore, $\theta=(p^{G})^{-1}\circ \theta^{+}\circ p^{F}$, if we put $\sigma_\theta=: (p^{G})^{-1}\circ \theta^{+}$, then  we have the result.
\end{proof}


 In what follows, $\mathcal{A}$ is a subcategory of $\mathcal{SET}$  or  $\mathcal{GRP}$, such that if $F\in \mathbf{Psh}(X, \mathcal{A})$, then $F^{+}\in \mathbf{Sh}(X, \mathcal{A})$ (a necessary condition for this is that $\mathcal{A}$ be closed under taking products and subobjects). $\mathcal{C}$ will denote a reflective subcategory of $\mathcal{A}$.\\

 Let $\mathcal{C}$ be a reflective subcategory of $\mathcal{A}$.
 
\begin{teo}\label{303} Suppose that $\mathcal{C}$ is closed under taking  products and subobjects. Then $\mathbf{Sh}(X,\mathcal{C})$ is a reflective subcategory of $\mathbf{Psh}(X, \mathcal{A})$.
    
\end{teo}

\begin{proof}

Given the assumptions on $\mathcal{C}$, we have that $(r_\mathcal{C}F)^+\in \mathbf{Sh}(X, \mathcal{C})$. We will prove that 

$$\xymatrix{F\ar[rr]^{\theta^{F,\mathcal{C}}} &   &   r_\mathcal{C}F\ar[rr]^{p^{r_\mathcal{C},F}} &   &  (r_\mathcal{C}F)^{+}}$$

is the $\mathbf{Sh}(X, \mathcal{C})-$reflection of $F$.  Indeed, let $G\in \mathbf{Sh}(X, \mathcal{C})$ and let $\theta\colon F\longrightarrow  G$  be a natural transformation, then we have the following commutative diagram

$$\xymatrix{ F\ar[rrdd]^{\theta}\ar[dd]_{\theta^{F, \mathcal{C}}}&   &    \\  &   &\\
             r_\mathcal{C}F\ar[rr]^{\mathcal{C}(\theta)} \ar[dd]_{p^{r_\mathcal{C}F}}
               &   &G \\ &    &\\
             (r_\mathcal{C}F)^{+},\ar[rruu]_{\sigma_{\mathcal{C}(\theta)}} &   &\\
             }$$ which yields the desired result.

\end{proof}
 
 Although Theorem \ref{303} is given for subcategories of $\mathcal{SET}$ or $\mathcal{GRP}$ it can be extended to categories with products, co-limits, and where bijective morphisms are isomorphisms. For example, the category $\mathbf{Por}(\leq)$ of preordered sets to which the technique used in Section \ref{1932} can be applied.

Note that the realization of the sheafification of $F \in \mathbf{Psh}(X, \mathcal{A})$ depends on the existence of colimits and products, closure under subobjects, and the property that bijective morphisms are isomorphisms. This is the case, for example, for categories such as preordered sets, rings, semigroups, and modules. However, this does not hold in general for topological categories, since in such categories bijective morphisms are not necessarily isomorphisms.
\begin{ejem}
Let $\mathcal{A}=\mathcal{GRP}$ and $\mathcal{C}=\mathcal{AB}$ the category of the abelian  groups. It is known that $\mathcal{AB}$ is a reflective subcategory of $\mathcal{GRP}$ which are closed under taking products and subgroups. Theorem \ref{303} implies that $\mathbf{Sh}(X, \mathcal{AB})$ is a reflective subcategory of $\mathbf{Psh}(X, \mathcal{GRP})$. This result holds if we replace topological groups with topological semigroups with open shifts (see \cite{arzusa2025group}).
\end{ejem}

\begin{ejem}
    Let $\mathcal{ABT}$ be the fully subcategory  of $\mathcal{AB}$. $\mathcal{ABT}$ is a subcategory of $\mathcal{AB}$ which is closed under taking small products and subgroups, therefore,  $\mathbf{Sh}(X, \mathcal{ABT})$ is a reflective subcategory of $\mathbf{Psh}(X, \mathcal{AB})$.
\end{ejem}

\begin{ejem}
    Let $\mathcal{SEM}$ the category of semigroups and $\mathcal{CSEM}$ the full subcategory of cancellative semigroups  (see \cite{arzusa2025group}).  $\mathcal{CSEM}$ is a reflective subcategory that satisfies the hypotheses of the  Theorem \ref{303}. Thus, $\mathbf{Sh}(X, \mathcal{CSEM})$ is a reflective subcategory of $\mathbf{Psh}(X, \mathcal{SEM})$.
\end{ejem}

\begin{ejem}
    Let $\mathbf{Por}(\leq)$ the category of preordered sets. If  $\mathbf{Or}(\leq)$ is the full subcategory of ordered sets. Then $\mathbf{Or}(\leq)$ is a reflective subcategory of $\mathbf{Por}(\leq)$ that satisfies  the hypotheses of the  Theorem \ref{303}. Thus, $\mathbf{Sh}(X, \mathbf{Or}(\leq))$ is a reflective subcategory of $\mathbf{Psh}(X, \mathbf{Por}(\leq))$.
\end{ejem}

Let $\mathcal{C}$ be  a reflective subcategory of a category $\mathcal{A}$. We say that the reflection functor $r_\mathcal{C} \colon \mathcal{A}\longrightarrow \mathcal{C}$ preserves small products if there exists an isomorphism $w\colon r_\mathcal{C}\prod_{x\in I}A_x\longrightarrow \prod_{x\in I} r_\mathcal{C}A_x$ such that the following diagram 

$$\xymatrix{r_\mathcal{C} \prod_{x\in I}A_i\ar[rrrr]^{w} &    &   &  &\prod_{x\in I}r_\mathcal{C}A_x\\
 &   &   &   &\\
 &  &  \prod_{x\in I} A_x\ar[uull]^{r_\mathcal{C}^{\prod_{x\in I}A_x}}\ar[uurr]_{\prod_{x\in I}r_\mathcal{C}^{A_x}}&   &}$$ commutes (see \cite{husek1987preservation}).\\

The following definition (equalizer) is very important because characterizes the concept of sheaf.
\begin{dfn}
   
Let $\mathcal{A}$ be a category (not necessarily a concrete category) and let $f,g : A \to B$ be morphisms in $\mathcal{C}$.  
An \emph{equalizer} of $f$ and $g$ is a morphism $e : E \to A$ such that:

\begin{enumerate}
    \item $f \circ e = g \circ e$,
    \item for any object $M$ and any morphism $h : M \to A$ with 
    $f \circ h = g \circ h$, there exists a unique morphism 
    $u : X \to E$ such that $h = e \circ u$.
\end{enumerate}
\end{dfn}

Let  $\mathcal{C}$ be a reflective subcategory of $\mathcal{A}$ (not necessarily concrete categories). If $r_\mathcal{C}^{A}$ is a monomorphism for every $A\in \mathcal{A}$ and $e\colon E \to A$ is an equalizer of $f,g\in Hom_{\mathcal{A}}(A,B)$ in $\mathcal{A}$, then $r_\mathcal{C}(e)\colon r_\mathcal{C}E\to r_\mathcal{C}A$ is an equalizer of $r_\mathcal{C}(f),r_\mathcal{C}(g)\colon r_\mathcal{C}A\to r_\mathcal{C}B$.\\

Let $F\in \mathbf{Psh}(X, \mathcal{A})$ and let  $\mathcal{C}$ be a reflective subcategory of $\mathcal{A}$, where  $\mathcal{A}\in \{\mathcal{SET}, \mathcal{GRP}\}$. Let $U$ be open in $X$ and $U=\bigcup_{i\in I}U_i$ be an open cover of $U$. Let us define $\mathbf{a}^{F}\colon F(U)\longrightarrow \prod_{i\in I}F(U_i)$, by $\mathbf{a}^F(s)=(F_{U_i}^{U}(s))_{i\in I}$, for each $s\in F(U)$. Let us define $\mathbf{c}^F\colon\prod_{i\in I}F(U_I)\longrightarrow \prod_{(i,j)\in I\times I}F(U_i\cap U_j)$ by $\mathbf{c}^F((s_i)_{i\in I})=(F_{U_i\cap U_i}^{U_j}(s_j))_{(i,j)\in I\times I}$, for each $(s_i)_{i\in I}\in \prod_{i\in I}F(U_i)$.  Let us define $\mathbf{b}^F\colon\prod_{i\in I}F(U_i)\longrightarrow \prod_{(i,j)\in I\times I}F(U_i\cap U_j)$ by $\mathbf{b}^{F}((s_i)_{i\in I})=(F_{U_i\cap U_i}^{U_j}(s_i))_{(i,j)\in I\times I}$, for each $(s_i)_{i\in I}\in \prod_{i\in I}F(U_i)$. Then $F$ is a sheaf if and only if $\mathbf{a}^F$ is an equalizer of $(\mathbf{b}^F, \mathbf{c}^F)$ (see \cite{tennison1975sheaf}).\\
Suppose that $\mathcal{C}$ is a reflective subcategory of $\mathcal{A}$ and the funtor induced by reflection, $r_\mathcal{C}$, preserves small products. Suppose that $r_\mathcal{C}A$ is a monomorphism for each $A\in Ob(\mathcal{A})$, therefore, $r_\mathcal{C}$ preserves equalizer. Then if $F$ is a sheaf, $r_\mathcal{C}(\mathbf{a}^F)\colon r_\mathcal{C}F(U)\longrightarrow \prod_{i\in I}r_\mathcal{C}F(U_i) $ is an equalizer of  $r_\mathcal{C}(\mathbf{c}^F)\colon\prod_{i\in I}r_\mathcal{C}F(U_i)\longrightarrow \prod_{(i,j)\in I\times I}r_\mathcal{C}F(U_i\cap U_j)$ and   $r_\mathcal{C}(\mathbf{b}^F)\colon\prod_{i\in I}r_\mathcal{C}F(U_i)\longrightarrow \prod_{(i,j)\in I\times I}r_\mathcal{C}F(U_i\cap U_j)$. This means $r_\mathcal{C}F$ is a sheaf. From Proposition \ref{2331} and Theorem \ref{253}, we have the following result.

\begin{teo}\label{3031}Suppose that $\mathcal{A}$ and $\mathcal{C}$ have equalizers and $\mathcal{A}$ is closed under taking subobjects. Then if $r_\mathcal{C}$ preserves small products and $r_\mathcal{C}^{A}$ is a monomorphism for any $A\in Ob(\mathcal{A})$, then $\mathbf{Sh}(X,\mathcal{C})$ is a reflective subcategory of $\mathbf{Psh}(X, \mathcal{A})$.
    
\end{teo}

\begin{proof} Let $F\in \mathbf{Psh}(X, \mathcal{A})$, then 
    under our assumptions we have that $r_\mathcal{C}(F^{+})\in \mathbf{Sh}(X, \mathcal{C})$. We will prove that 

    $$\xymatrix{  F\ar[rr]^{p^{F}}&   &   F^{+}\ar[rr]^{\theta^{F^{+},\mathcal{C}}} &    & r_\mathcal{C}(F^{+})}$$ is the $\mathbf{Sh}(X, \mathcal{C})$-reflection of $F$. Indeed, let $G\in Sh(X,\mathcal{C})$ and let $\theta F\longrightarrow G$ a natural transformation. Then we have the following commutative diagram

 $$\xymatrix{ F\ar[rrdd]^{\theta}\ar[dd]_{p^{F}}&   &    \\  &   &\\
             F^{+}\ar[rr]^{\sigma_\theta} \ar[dd]_{\theta^{F{{+},\mathcal{C}}}}
               &   &G \\ &    &\\
             r_\mathcal{C}(F^{+}),\ar[rruu]_{\mathcal{C}(\sigma_\theta )}&   &\\
             }$$ then we have the result.
\end{proof}

Suppose  $\mathcal{C}$ and $r_\mathcal{C}$ satisfy the hypotheses of theorems  \ref{303} and \ref{3031}. If $F\in \mathbf{Psh}(X, \mathcal{A})$, then $r_\mathcal{C}(F^{+})$ and $(r_\mathcal{C}F)^{+}$ are $\mathbf{Sh}(X,\mathcal{C})$-reflections of $F$, therefore,  $r_\mathcal{C}(F^{+})$ is isomorphic to $(r_\mathcal{C}F)^{+}$. Thus, we have the following result.

\begin{teo} Let $\mathcal{A}$ be  a category closed under taking subobjects.
    Let $\mathcal{C}$ a reflective subcategory of $\mathcal{A}$ with equalizers, such that the reflection functor preserves products and $r_\mathcal{C}^{A}$ is a monomorphism for each $A\in Ob(\mathcal{A})$. Then $(r_\mathcal{C}F)^{+}$ is isomorphic to $r_\mathcal{C}(F^{+})$.
\end{teo}

\begin{ejem}
    Let $\mathcal{ABC}$ the category of the abelian cancellative semigroups (\cite[Section 1.2]{hernandez2023}). The category $\mathcal{AB}$ of the abelian groups is a reflective subcategory of $\mathcal{ABC}$. The reflection funtor $r_{\mathcal{AB}}$ preserves small products and   subcategory is reflective, preserves small products and $r_{\mathcal{AB}}^{A}$ is a monomorphism for each $A\in Ob(\mathcal{ABC})$. Therefore, for each $F\in \mathbf{Psh}(X, \mathcal{ABC})$,  $(r_{\mathcal{AB}}F)^{+}$ is isomorphic to $r_{\mathcal{AB}}(F^{+})$.
\end{ejem}

Let $\mathcal{KH}$ be the full subcategory of $\mathcal{TOP}$ consisting of compact Hausdorff spaces. This category is closed under small products, and every bijective morphism is an isomorphism. However, it is not closed under subobjects, which prevents us from applying the technique used so far to construct the sheafification.

Given $F \in \mathbf{Psh}(X, \mathcal{KH})$, we use $F^{+} \in \mathbf{Psh}(X, \mathcal{TOP})$ to construct the sheafification of $F$ in $\mathbf{Sh}(X, \mathcal{KH})$. Section \ref{1932} guarantees that constructing $F^{+}$ for presheaves with values in $\mathcal{TOP}$ is the same as in $\mathcal{SET}$.\\
It was shown that $F^{+} \in \mathbf{Sh}(X, \mathcal{TOP})$ if we endow each $F^{+}(U)$ with the subspace topology induced from $\prod_{x \in U} F_x$, where the latter is equipped with the product topology, for every open set $U \subseteq X$.\\
Define $F^{\ast}(U) = \overline{F^{+}(U)}$, for each open set $U \subseteq X$, where the closure is taken in $\prod_{x \in U} F_x$. Then $F^{\ast}(U) \in Ob(\mathcal{KH})$.\\
If $V \subseteq U$ are open sets in $X$, define $(F^{\ast})_V^U((s_x)_{x \in U}) = (s_x)_{x \in V}$. It follows from the definition that $(F^{+})_V^U$ is the restriction map of $(F^{\ast})_V^{U}$. By continuity of $(F^{\ast})_V^U$, we have $$(F^{\ast})_V^U(F^{\ast}(U)=(F^{\ast})_{V}^{U}\overline{F^+(U)}\subseteq \overline{(F^{\ast})_V^U(F^{+}(U))}=\overline{(F^{+})_V^{U}(F^{+}(U)}\subseteq \overline{F^{+}(V)}=F^{\ast}(V).$$

The following result is very well known in $\mathcal{TOP}$, it can be consulted in \cite{engelking1977general}.

\begin{prop} \label{54} \begin{itemize}
    
\item[i)] Let $f,g \colon Z\longrightarrow Y$ be continuous maps, where $Z$ and $Y$ are Hausdorff topological spaces. If $A\subseteq Z$ and $f(x)=g(x)$ for each $x\in A$, then $f(x)=g(x)$ for each $x\in \overline{A}$.

\item[ii)] Let $f\colon Z\longrightarrow Y$  be a continuous bijection, where $Z$ and $Y$ compact Hausdorff topological spaces. Then $f$ is a homeomorphism (it means $f$ is  an isomorphism in $Hom_\mathcal{TOP}(X,Y)$.

\item[iii)] If $X$ is Hausdorff and $B$ is compact in $X$, then $B$ is closed in $X$.

\end{itemize}
\end{prop}

Proposition \ref{54} $i)$ implies that the following diagram 

$$\xymatrix{ F(U)\ar[dd]_{F_V^{U}}\ar[rr]^{p_U} &    &   F^{\ast}(U)\ar[dd]^{(F^{\ast})_V^{U}}\\
&      &\\
F(V)\ar[rr]^{p_V} &    &  F^{\ast}(V) } $$ commutes. Which proves that $p^{F}\colon F\longrightarrow F^{\ast}$ is a  natural transformation.

Now, if $U$ is an open set in $X$, then $\mathbf{a}^{F^{+}}$ is the restriction of $\mathbf{a}^{F^\ast}$ in $F^{+}(U)$. The same applies to $\mathbf{b}^{F^{+}}$ and $\mathbf{c}^{F^{+}}$. Given that $\mathbf{a}^{F^{+}}$ is an equalizer of $(\mathbf{b}^{F^{+}}, \mathbf{c}^{F^{+}})$, Proposition \ref{54} $i)$ implies that $\mathbf{a}^{F^{\ast}}$ is an equalizer of $(\mathbf{b}^{F^{\ast}}, \mathbf{c}^{F^{\ast}})$, it means $F^{\ast}\in \mathbf{Sh}(X,\mathcal{KH})$. \\

If $F$ is $\mathbf{Sh}(X\mathcal{KH})$, then, of the proof of the Lemma \ref{223}, it follows that $p_U\colon F(U)\longrightarrow F^+(U)$ is a continuous bijection, therefore, $F^{+}(U)$ is compact, thus, by Proposition \ref{54} $iii)$, we have that $F^{+}(U)=F^\ast(U)$. So that, by Proposition \ref{54} $ii)$, $p_U\colon F(U)\longrightarrow F^{\ast}(U)$ is an isomorphism in $\mathcal{KH}$. A similar line of reasoning to that used in the proof of Theorem \ref{303} leads us to say that $F^{\ast}$ is the sheafification of $F$ with values in $\mathcal{KH}$.\\
We have constructed a sheafification, not only in $\mathcal{KH}$, but also in every subcategory $\mathcal{A}$ of $\mathcal{KH}$ or a subcategory of a topological algebraic category (such as topological semigroups or topological rings) with compact objects, closed by taking closed subobjects.\\The following Example exhibits an example of sheafification in categories with algebraic topological structures.

\begin{ejem}
    If $\mathcal{KS}$ is the category of the Hausdorff compact topological semigroups and $\mathcal{KG}$ the full subcategory of the Hausdorff compact topological groups. Then $\mathbf{Sh}(X,\mathcal{KG})$ is a reflective subcategory of $\mathbf{Psh}(X, \mathcal{KS})$.
\end{ejem}

\begin{proof}
    Let  $\mathcal{KSC}$ be the full subcategory of $\mathcal{KS}$  of the cancellative Hausdorff  compact topological semigroups, Which coincides with $\mathcal{KG}$ (see \cite{hernandez2021note}). These categories are closed under taking products and closed subobjects. Since the closure of a subgroup of a topological group is a subgroup, therefore, if $\mathcal{A}$ is a subcategory of $\mathcal{KG}$ closed under taking closed subgroups, then $F^{\ast}\in \mathbf{Sh}(X,\mathcal{KH})$ whenever $F\in \mathbf{Psh}(X,\mathcal{KH}). $ \\If $F\in \mathbf{Psh}(X,\mathcal{KSC})$, since $\mathcal{KSC}$ is a reflective subcategory of $\mathcal{KS}$, then $r_\mathcal{KSC}F\in \mathbf{Psh}(X, \mathcal{KSH})$. From the fact that $\mathcal{KSC}$ is closed under taking closed subobjects, then $(r_\mathcal{KG}F)^{\ast}\in \mathbf{Sh}(X,\mathcal{KG})$. A similar line of reasoning to that given in the proof of Theorem \ref{303} leads us to say that $(r_\mathcal{KG}F)^{\ast}$ is the $\mathbf{Sh}(X,\mathcal{GK})-$reflection on $F$. 
\end{proof}

Let $\mathcal{KGT}$ the full subcategory of $\mathcal{KG}$ of torsion-free compact topological groups. $\mathcal{KGT}$ is a reflective subcategory of  $\mathcal{KGT}$ closed under taking closed subgroups. Therefore, we have that $\mathbf{Sh}(X,\mathcal{KGT})$ is a reflective subcategory of $\mathbf{Psh}(X,\mathcal{KG})$. We can find a similar result for  full subcategories  of $\mathcal{KG}$  of abelian groups and totally disconnected groups. \\




\bibliographystyle{plain}
\bibliography{Biblio}

@book{agore2023first,
  title={A first course in category theory},
  author={Agore, Ana},
  year={2023},
  publisher={Springer}
}

@book{kashiwara2006,
   author     = "Kashiwara, M. and Schapira, P.",
   title      =    "Categories and sheaves",
   adress     =  "Berlin",
   publisher  =  "Springer",
   year       =     "2006"
}

@book{engelking1977general,
author="Engelking, Ryszard and others",
  title={General topology},
  adress="Berlin",
   volume={60},
    publisher={PWN Warszawa},
  year={1977},
 
}

@book{tennison1975sheaf,
 author="Tennison, B. R.",
  title= "Sheaf theory",
  adress= "Cambridge",
 publisher = "Cambridge University Press",
 year= "1975"
}

@book{maclane2012sheaves,
 author  =  "MacLane, S. and Moerdijk, I.",
  title   =   "Sheaves in geometry and logic: A first introduction to topos theory",
 year={2012},
  publisher={Springer Science \& Business Media}
}

@article{bunk2023homotopy,
  title={Homotopy Sheaves on Generalised Spaces},
  author={Bunk, Severin},
  journal={Applied Categorical Structures},
  volume={31},
  number={6},
  pages={49},
  year={2023},
  publisher={Springer}
}

@article{barnes2022equivariant,
  title={Equivariant sheaves for profinite groups},
  author={Barnes, David and Sugrue, Danny},
  journal={Topology and its Applications},
  volume={319},
  pages={108215},
  year={2022},
  publisher={Elsevier}
}

@article{cortes2023reflective,
  title={Reflective and coreflective subcategories},
  author={Cort{\'e}s-Izurdiaga, Manuel and Crivei, Septimiu and Saorin, Manuel},
  journal={Journal of Pure and Applied Algebra},
  volume={227},
  number={5},
  pages={107267},
  year={2023},
  publisher={Elsevier}
}

@article{goodbody2026approximable,
  title={Approximable Triangulated Categories and Reflexive DG-categories: I. Goodbody},
  author={Goodbody, Isambard},
  journal={Applied Categorical Structures},
  volume={34},
  number={3},
  pages={19},
  year={2026},
  publisher={Springer}
}

@article{barbosa2023sheaf,
  title={Sheaf representation of monoidal categories},
  author={Barbosa, Rui Soares and Heunen, Chris},
  journal={Advances in Mathematics},
  volume={416},
  pages={108900},
  year={2023},
  publisher={Elsevier}
}

@article{hernandez2020reflections,
  author  =  "Hernández-Arzusa, J.C. and Hernandez, S.",
  title   =    "Reflections in topological algebraic structures",
  
  journal= "Topology and its Applications",
  volume=  "281",
pages=   "107204",
  year= "2020"
}

@article{godement1973topologie,
  title={Topologie alg{\'e}brique et th{\'e}orie des faisceaux. Hermann, Paris, 1973. Troisieme {\'e}dition revue et corrig{\'e}e, Publications de l’Institut de Math{\'e}matique de l’Universit{\'e} de Strasbourg, XIII},
  author={Godement, Roger},
  journal={Actualit{\'e}s Scientifiques et Industrielles},
  volume={1252},
  year={1973}
}

@article{castillo2021dual,
  title={El dual de la reflexi{\'o}n de un grupo topol{\'o}gico.},
  author={Castillo, Adriana C and Hern{\'a}ndez Arzusa, Julio César},
  journal={Revista Integraci{\'o}n, temas de matem{\'a}ticas},
  volume={39},
  number={1},
  pages={23--40},
  year={2021}
}

@article{gray1965sheaves,
 author   =   "Gray, J.W.",
  title  =  "Sheaves with values in a category",
 
  journal  = "Topology",
  volume  =  "3",
 pages  =  "1--18",
  year  ="1965"
  }

@inproceedings{hernandez2021note,
  title={A note on locally compact subsemigroups of compact groups},
  author="Hern{\'a}ndez Arzusa, Julio César and Hofmann, Karl H",
  booktitle="Semigroup Forum",
  volume={103},
  number={1},
  pages={291--294},
  year={2021}
}

@article{tkachenko2014axioms,
  author={Tkachenko, Mikhail},
  title={Axioms of separation in semitopological groups and related functors},
journal={Topology and its Applications},
  volume={161},
  pages={364--376},
  year={2014},
  publisher={Elsevier}
}

@article{husek1987preservation,
  title={Preservation of products by functors close to reflectors},
  author={Husek, M and De Vries, J},
  journal={Topology and its Applications},
  volume={27},
  number={2},
  pages={171--189},
  year={1987},
  publisher={Elsevier}
}

@article{arzusa2025group,
  title={The group of characters of a pseudocompact locally compact semitopological semigroup},
  author="Hernández Arzusa, Julio C{\'e}sar",
  journal={Applied General Topology},
  volume={26},
  number={2},
  pages={681--689},
  year={2025}
}

@book{hernandez2023,
author="Hern{\'a}ndez Arzusa, Julio C{\'e}sar",
  title={Una introducci{\'o}n a grupos y semigrupos topol{\'o}gicos.},
   year={2023},
  publisher={Universidad de Cartagena}
}

\end{document}